\documentclass[a4paper,11pt,reqno]{amsart}

\usepackage[colorlinks=true, pdfstartview=FitV, linkcolor=blue,
citecolor=blue]{hyperref}

\usepackage{amssymb,amsmath,amscd,bbm,bm}
\usepackage{graphicx,subcaption}
\usepackage{a4wide,rotating}
\usepackage{tikz}
\usetikzlibrary{graphs,arrows,shapes,calc}
\newcommand{\tmark}[1]{\tikz[remember picture, overlay] \node(#1){};}

\tikzset{every picture/.style={line width=0.75pt}}

\newtheorem{theorem}{Theorem}
\newtheorem{prop}[theorem]{Proposition}
\newtheorem{lemma}[theorem]{Lemma}
\theoremstyle{definition}

\newtheorem{remark}[theorem]{Remark}

\newcommand{\ts}{\hspace{0.5pt}}

\newcommand{\RR}{\mathbb{R}\ts}

\newcommand{\ZZ}{\mathbb{Z}}
\newcommand{\NN}{\mathbb{N}}
\newcommand{\QQ}{\mathbb{Q}}
\newcommand{\XX}{\mathbb{X}}

\newcommand{\cA}{\mathcal{A}}

\newcommand{\cL}{\mathcal{L}}

\newcommand{\cT}{\mathcal{T}}

\newcommand{\vL}{\varLambda}
\newcommand{\vrho}{\varrho}

\newcommand{\dd}{\,\mathrm{d}}

\newcommand{\oplam}{\mbox{\Large $\curlywedge$}}
\newcommand{\bs}{\boldsymbol}

\newcommand{\card}{\ts\mathrm{card}\ts}

\newcommand{\exend}{\hfill$\Diamond$}

\newcommand{\myfrac}[2]{\frac{\raisebox{-2pt}{$#1$}}
	{\raisebox{0.5pt}{$#2$}}}

\makeatletter
\providecommand*{\bigcupdot}{%
  \mathop{%
    \vphantom{\bigcup}%
    \mathpalette\@bigcupdot{}%
  }%
}
\newcommand*{\@bigcupdot}[2]{%
  \ooalign{%
    $\m@th#1\bigcup$\cr
    \sbox0{$#1\bigcup$}%
    \dimen@=\ht0 %
    \advance\dimen@ by -\dp0
    \sbox0{\scalebox{\@dotscale{#1}}{$\m@th#1\cdot$}}%
    \advance\dimen@ by -\ht0
    \dimen@=.5\dimen@
    \hidewidth\raise\dimen@\box0\hidewidth
  }%
}
\newcommand{\@dotscale}[1]{%
  \ifx#1\displaystyle 1.7\else\ifx#1\textstyle 1.4\else
  \ifx#1\scriptstyle 1.2\else 1.1\fi\fi\fi
}
\makeatother

\makeatletter
\renewcommand{\@captionfont}{\small}
\makeatother

\DeclareMathOperator{\dens}{dens}
\DeclareMathOperator{\vol}{vol}

\DeclareFontFamily{U}{mathx}{}
\DeclareFontShape{U}{mathx}{m}{n}{<-> mathx10}{}
\DeclareSymbolFont{mathx}{U}{mathx}{m}{n}
\DeclareMathAccent{\widecheck}{0}{mathx}{"71}

\DeclareMathAlphabet{\mathmybb}{U}{bbold}{m}{n}

\newcommand{\defeq}{\mathrel{\mathop:}=}

\begin{document}
	
\title[Exact renormalisation for patch frequencies in inflation systems]{Exact renormalisation for patch frequencies in inflation systems}
	
\author{Jan Maz\'{a}\v{c}}
\address{Fakult\"at f\"ur Mathematik, Universit\"at Bielefeld,
	\newline \indent Postfach 100131, 33501 Bielefeld, Germany}
\email{jmazac@math.uni-bielefeld.de}

\begin{abstract}
This note provides an explicit way of calculating the patch frequencies in geometric realisations of primitive substitutions using exact renormalisation relations. Further, we profit from these results to obtain the patch frequencies in the symbolic case as well as in other suspensions (under mild assumptions on the substitution). We illustrate this procedure on the Fibonacci example. 
\end{abstract}
	

\maketitle

\section{Introduction}
Subshifts over finite alphabets $\cA$ and their corresponding dynamical systems have been studied extensively in the last decades. A particularly well-understood class arises from fixed points of substitution rules. It is a classic result that, in the case of primitive substitutions, the subshift (or the symbolic hull) is minimal and uniquely ergodic. For non-primitive substitutions, the situation becomes more complicated as the dynamical system is no longer minimal and more ergodic measures can exist (see \cite{CorSol} for a detailed study in this direction, and \cite{BKMS} for a more general discussion for Bratelli diagrams). 

While the existence of the \emph{patch frequency measure} for primitive systems is clear, the question of how to explicitly compute the relative frequency of a given patch (which need not be a word) remains unclear. For the symbolic setting, one can derive the induced substitution on $m$ letters, see \cite[Sec.~5.4]{Que} for details. 
Unfortunately, this approach provides frequencies for cylinder sets only, and cannot be used to determine whether a given patch is legal. 
For example, we cannot decide the legality of a pattern consisting of letter $a^{}_{1}$ at the zeroth position, $a^{}_{2}$ at $\ell$-th position and $a^{}_{3}$ at $m$-th position, and compute its relative frequency. Unless the underlying symbolic hull does not belong to a special class (see \cite{Bal12,Dek92,Frid,Pol} for explicit results), to the best of our knowledge, there is no general procedure that allows one to answer the question for a primitive aperiodic substitution rule. 

The aim of this note is to provide a method to compute the exact patch frequency of any given patch (not only a subword) and (thus) deciding its legality, provided we work with systems arising from primitive substitutions. 

The paper is organised as follows. We first recall the relevant notions from symbolic dynamics and the ways that allow one to pass from a symbolic to a geometric setting. 
We use a self-similar tiling (a set of its control points) from a transversal of the geometric hull and show how to compute the relative frequency of \emph{any} given patch. In particular, we derive exact renormalisation equations for the patch frequencies (Theorem~\ref{thm:patch_renorm_subst}) and show that they possess a unique solution (Theorem~\ref{thm:uniqueness}). 
Then, we illustrate how to pass from geometric to symbolic setting in the Fibonacci example, and we provide a way of deriving the patch frequency for the symbolic setting and for other suspensions under some mild assumptions on the roof function and the substitution (Theorems~\ref{thm:geom_to_symb} and \ref{thm:geom_to_arb}).

\section{Subshifts and their suspensions}
Consider a primitive aperiodic substitution $\varrho$ on a finite alphabet $\cA$ with $n$ letters and its fixed point~$w$, i.e., $w\in\cA^{\ZZ}$ with $w = \varrho^k(w)$ for some $k\in \NN$. The shift operator $S:\cA^{\ZZ} \rightarrow \cA^{\ZZ}$ acts via $(Sw)^{}_{i} \, =\, w^{}_{i+1}$. On $\cA^{\ZZ}$, we consider the product topology. This allows us to define the hull of $w$ as 
\[\XX(w) \, = \, \overline{\{S^i w \ \mid \ i\in \ZZ \} }, \]
a compact space, which is independent of the choice of the fixed point $w$, so we can write $\XX(w) = \XX(\varrho)$ and speak about a symbolic hull of a substitution $\varrho$. 
With the shift action $S$, it becomes a topological dynamical system $(\XX(\varrho), \, \ZZ)$, which is minimal \cite[Thm.~4.1]{BGr}, and strictly ergodic under the $\ZZ$-action, see \cite[Thm.~4.3]{BGr} for details. 
The unique ergodic measure $\mu$ on $(\XX(\varrho),\ \ZZ)$ is the patch frequency measure defined via cylinder sets.  Recall that, for a finite word $u=u^{}_0\dots u^{}_{\ell}$, $u^{}_{i}\in \cA$, the cylinder set $[u] \subset \XX(\vrho)$ is defined by
\[ [u] \, \defeq \, \{ x\in \XX(\vrho) \ \mid \ x^{}_{i} = u^{}_{i} \ \mbox{for } 0\,\leqslant \, i \, \leqslant \, \ell \}, \]
and for its measure, we have 
\[ \mu\bigl([u]\bigr) \, \defeq \, \mbox{relative frequency of } u \mbox{ in } w.\]
Due to the strict ergodicity of $\XX(\vrho)$, the frequency does not depend on the choice of the element from the hull. 

Now, consider a non-negative function $f:\cA \rightarrow \RR^{+}$. 
A standard construction turns the dynamical system $(\XX(\vrho),\, \ZZ )$ into a flow $(\XX(\vrho)^{(f)},\,\RR)$, usually called a suspension flow, or a~flow under a~function with base $(\XX(\vrho),\, \ZZ )$. Here, 
\[\XX(\vrho)^{(f)} \, = \, \bigl(\XX(\vrho) \times \RR \bigr) \Bigl/ \sim \, = \, \{(x,s) \ \mid x\in \XX(\vrho),\ 0 \leqslant s < f(x^{}_{0}) \}, \]
where we identify $(x, f(x^{}_{0})) \sim (S(x),0)$, with $x = \dots x^{}_{-1}\, x^{}_{0}\, x^{}_{1} \dots\ $. The $\RR$-action is given, for all $t\in\RR$, by $V^{}_{t}: (x,s) \mapsto (x,s+t) \mod \sim$, for any $x\in \XX(\vrho)$ and $s\in\RR$. 
Explicitly, the action reads 
\[ V^{}_{t}(x,s) \, = \, \begin{cases}
    (x, s+t), & \mbox{if } 0\leqslant s+t <f(x^{}_{0}), \\
    (S(x), s+t-f(x^{}_{0})), & \mbox{if } 0\leqslant s+t -f(x^{}_{0}) < f(S(x^{}_{0})) = f(x^{}_{1}),\\
    (S^2(x), s+t-f(x^{}_{0})-f(x^{}_{1})), & \mbox{if } 0\leqslant s+t -f(x^{}_{0}) -f(x^{}_{1})< f(x^{}_{2}),\\
    & \vdots
\end{cases}\]

More details about the suspension flows can be found for example in \cite[Ch.~11.1.]{CFS}, in \cite[Lemma~9.23]{EW}, or, in the context of substitutions, in \cite{BS,ST}. 
The probability measure $\mu$ extends to the suspension flow naturally, giving a probability measure $\mu^{}_{f}$, such that 
\[ \mu^{}_{f}(A,I\ts) \, \defeq \, \myfrac{ \mu(A) }{\int^{}_{\XX(\vrho)}f \dd \mu} \,\int^{}_{I} \dd s \]
holds for all $A\subset \XX(\vrho)$ and $I\subset \RR$. 

There is a natural interpretation of this dynamical system. Consider a tiling of $\RR$, i.e., a~collection $\mathcal{T}=\{T^{}_{i} \}^{}_{i\in\ZZ}$ of tiles $T^{}_{i}$ that cover $\RR$ and satisfy $\mathrm{int}(T^{}_{i}) \cap \mathrm{int}(T^{}_{j}) = \varnothing$ for $i\neq j$. 
The tiles $T^{}_{i}$ are translates of prototiles $I^{}_{a} = [0,\, f(a) )$, with $a\in \cA$. The tiles $T^{}_{i}$ are chosen so that $T^{}_{i}$ is a translate of $I^{}_{w^{}_{i}}$. 
In other words, we replaced the letters from $\cA$ in $w$ by shifted intervals, such that the left endpoint of the tile $T^{}_{0}$ coincides with the origin. 
This yields a dynamical system with an $\RR$-action $(\XX(\mathcal{T}),\, \RR)$, where the $\RR$-action is the translation $U^{}_{t}: \{T^{}_{i} \}^{}_{i\in\ZZ} \mapsto \{T^{}_{i}- t \}^{}_{i\in \ZZ}$, and $\XX(\mathcal{T})$ is the hull of a tiling $\mathcal{T}$ given by the orbit closure of $\mathcal{T}$ in the local topology, see \cite[Sec.~5.1]{BGr}. 
This dynamical system is essentially the same as the suspension flow $(\XX(\varrho)^{(f)},\,\RR)$, meaning that they are topologically conjugate \cite{CS1}. 

If we, instead of the entire tiling $\mathcal{T}$, consider the left endpoints of each interval (and eventually with colours to distinguish different prototile type), denoting the resulting point set by $\vL$, we obtain another dynamical system $(\XX(\vL), \, \RR)$, which is locally derivable (in the sense of \cite{BSJ91}) from $(\XX(\cT), \, \RR)$ (and vice versa). This system is uniquely ergodic too, and the ergodic measure $\nu$ is \emph{again} given by the patch frequencies. The measure is induced by the patch frequencies on the \emph{transversal}, i.e., on the set 
\[\XX^{}_{0}(\vL) \, = \, \{X\in \XX(\vL) \ \mid \ 0\in X \}. \]
The transversal is a topological space isomorphic to $\XX(\vrho)$, but there is, in general, no group action that would turn it into a dynamical system. Nevertheless, it provides us with important information as we shall see later.

Among all choices of the function $f$, there are two prominent ones. Either $f$ is a constant function, and we essentially recover the symbolic case, or we consider $(f(a))^{}_{a\in\cA}$ to be the left Perron--Frobenius (PF) eigenvector of the substitution matrix $M^{}_{\vrho}$ of substitution $\vrho$. 
Recall that $(M^{}_{\vrho})^{}_{ab}$ counts the number of $a$'s in $\vrho(b)$. If we choose the smallest entry of the eigenvector to be 1, we speak about \emph{natural tile lengths}.
In such a case, we obtain an additional feature of the underlying structure, namely, the inflation symmetry (or the self-similarity) which manifests itself in an essential way on the transversal $\XX^{}_{0}(\vL)$. 
We note that the hull corresponding to the self-similar tiling $\cT$ is usually called the \emph{geometric hull} of the substitution $\vrho$, see \cite[Sec.~5.4]{BGr}. We will later rely on the self-similar structure of the tiling and its left control points.

Recall that the last two above-mentioned dynamical systems (the two choices of the function~$f$) \emph{need not} be topologically conjugate although when considered as topological spaces, they are always homeomorphic \cite[Thm.~1.1]{CS1}. This is generally true for any two choices of functions $f$ and $g$, and Clark and Sadun also provide a sufficient condition of the substitution $\vrho$ (namely, $M^{}_{\vrho}$ has only one eigenvalue of magnitude one or greater (counted with multiplicities)) for $(\XX(\vrho)^{(f)},\,\RR)$ and $(\XX(\vrho)^{(g)},\,\RR)$ to be topologically conjugate \cite[Cor.~3.2]{CS1}. 
Another approach (which generalises to higher dimensions) uses tools from algebraic topology and states that two such suspensions are topologically conjugate if the corresponding tilings are related by an asymptotically negligible shape change (for details, consult \cite{CS2,Kel,JS}).

\section{geometric setting and the renormalisation equations}
Denote by $(\ell^{}_{1}, \, \dots, \ell^{}_{n})$ the entries of the left PF eigenvector of $M^{}_{\vrho}$ to the PF eigenvalue $\lambda$. The substitution rule $\vrho$ induces a geometric inflation rule on tiles $I^{}_{i} = [0, \ell^{}_{i})$ with 
\[ 
\lambda I^{}_{i} \, = \, \bigcupdot_{j=1}^n \bigl(I^{}_j + t^{}_{ji} \, \bigr), 
\]
where $t^{}_{ji} \subset \RR$ are finite sets, $I^{}_j + t^{}_{ji}$ is the Minkowski sum (in particular, $I^{}_j + \varnothing = \varnothing$) and the sets $t^{}_{ji}$ are uniquely determined by the symbolic substitution rule $\vrho$. In other words, the geometric inflation rule inflates the tile $I^{}_{i}$ by a factor $\lambda$ and subdivides according to the letters in the word $\vrho(a^{}_{i})$. This naturally extends to all translates of the prototiles.
Consider the control points $\vL^{}_{i}\subset \RR$ of all tiles of type $i$ (always take the left endpoint) and denote $\underline{\vL} = (\vL^{}_{1}, \, \dots, \vL^{}_{n})$. Since the original word $w$ was invariant (without loss of generality) under the action of $\vrho$, the control points of the geometric inflation tiling obey the following induced inflation symmetry 
\[\vL^{}_i \, = \, \bigcupdot_{j=1}^{n} \,\bigl(  \lambda \vL+ t^{}_{ij} \bigr) \, = \, \bigcupdot_{j=1}^{n} \, \bigcupdot_{t \in t^{}_{ij}} 
\bigl( \lambda \vL^{}_j + t \bigr), \]
which holds for $1\leqslant i\leqslant n$. The set-valued matrix $(t^{}_{ij})^{}_{1\leqslant i,j\leqslant n}$ is called the \emph{displacement matrix} and fully determines the substitution (whereas the substitution matrix $M^{}_{\varrho} = (\card (t^{}_{ij}))^{}_{1\leqslant i,j \leqslant n}$ does not). For further details on such inflation systems, see \cite[Sec.~4.8 and 5.4]{Sing}

Now, consider the relative frequency of a patch consisting of $m$ tiles $a^{}_{1}, \, \dots, \, a^{}_{m}$ within distances $x^{}_{1}, \, \dots, \, x^{}_{m-1}$ with $x^{}_{i}\in \RR$, (i.e., the distance between the left endpoints of $a^{}_{i}$ and $a^{}_{i+1}$ is $x^{}_{i}$), which will be denoted by $\nu^{}_{a^{}_{1} \, \cdots \, a^{}_{m}}(x^{}_{1}, \, \dots, \, x^{}_{m-1})$.  We call this function the \emph{patch frequency function}.
Since the point set $\vL = \bigcup^{n}_{i=1} \vL^{}_{i}$ arises from a primitive substitution, $(\XX(\vL), \, \RR)$ is strictly ergodic \cite{Sol97} and patch frequencies exists uniformly for each element of $\XX(\vL)$. 

Set $x^{}_{0}=0$. Then, we can write
\begin{equation}
\label{eq:patch_freq_def}
\begin{split}
    \nu^{}_{a^{}_{1}\, \cdots \, a^{}_{m}}(x^{}_{1}, \, \dots, \, x^{}_{m-1}) &\, = \, \lim_{R\to \infty} \myfrac{1}{\card\bigl(\vL^{(R)} \bigr)} \, \card \Bigl( {\displaystyle \bigcap_{i=1}^{m}} \bigl( \vL^{(R)}_{a^{}_{i}} - \sum_{j=0}^{i-1}x^{}_{j}\bigr) \Bigr)\\[2pt]
    &\, = \, \myfrac{1}{\dens(\vL )}\,\dens\Bigl( {\displaystyle \bigcap_{i=1}^{m}} \bigl( \vL_{a^{}_{i}} - \sum_{j=0}^{i-1}x^{}_{j}\bigr) \Bigr),
\end{split}
\end{equation}
where $A^{(R)} \, \defeq \, \{x\in A \ \mid \ |x| \leqslant R\}$, and $\dens(A)$ denotes the density of a discrete set $A$, which is, in this case, well defined.

By definition, $\nu^{}_{a^{}_{1}\, \cdots \, a^{}_{m}}(x^{}_{1}, \, \dots, \, x^{}_{m-1}) \geqslant 0$, for any $m\geqslant2$, any $a^{}_{i} \in \cA$ and any $x^{}_{i} \in \RR$. Clearly, if $x^{}_{j} \notin \vL^{}_{a^{}_{j+1}} \!\! - \vL^{}_{a^{}_j}$ for some $j$, then $\nu^{}_{a^{}_{1}\, \cdots \, a^{}_{m}}(x^{}_{1}, \, \dots, \, x^{}_{m-1}) =0$. We summarise some other properties of the patch frequency function, which directly follow from its definition. 
\begin{prop}
\label{prop:patch_freq_fun_properties}
Let $\XX(\vL)$ be the geometric hull defined by a fixed point $\underline{\vL}$ of a primitive, aperiodic, geometric inflation rule induced by a substitution on $\cA$. Then, for any $m\geqslant 2$ and for any $a^{}_{1}, \, \dots, \, a^{}_{m} \in \cA$, the patch frequency functions $\nu^{}_{a^{}_{1}\, \cdots \, a^{}_{m}}$ defined by \eqref{eq:patch_freq_def} exist uniformly on $\XX(\vL)$ and are the same for any $\vL' \in \XX(\vL)$. Moreover, they satisfy the following. 
\begin{enumerate}
    \item[(i)] For all $m\geqslant 2$ and for any $a^{}_{1}, \, \dots, \, a^{}_{m} \in \cA$, we have 
    \[ \nu^{}_{a^{}_{1}\, \cdots \, a^{}_{m}}(x^{}_{1}, \, \dots, \, x^{}_{m-1}) \, = \, \nu^{}_{a^{}_{m}\, \cdots \, a^{}_{1}}(-x^{}_{m-1}, \, \dots, \, -x^{}_{1}). \]
    \item[(ii)] If a patch $P$ occurs with zero frequency, the patch frequency of any patch $P'$ that contains $P$ vanishes as well. 
    \item[(iii)] For any $A\in\cA$, we have 
\begin{align*}
    \nu^{}_{a^{}_{1}\, \cdots\, a^{}_{i}\, A \, A \, a^{}_{i+3}\, \cdots \, a^{}_{m}}&(x^{}_{1}, \, \dots,\,x^{}_{i}, \, 0, \, x^{}_{i+2},\, \dots, \, x^{}_{m-1}) \\ & =\, \nu^{}_{a^{}_{1}\, \cdots\, a^{}_{i}\, A \, a^{}_{i+3}\, \cdots \, a^{}_{m}}(x^{}_{1}, \, \dots,\,x^{}_{i}, \, x^{}_{i+2},\, \dots, \, x^{}_{m-1}) .
\end{align*}
\end{enumerate}
\qed
\end{prop}

There are further symmetries, including those connecting frequencies of patches and their subpatches. For example, the relation between frequencies of three- and four-tile patches can be stated as follows
\[\nu^{}_{a^{}_{1}\, a^{}_{2}\, a^{}_{3}} (x^{}_{1}, \, x^{}_{2}) \, = \, \sum_{a\in\cA}\Bigl( \sum_{\substack{s,\, t \ts \geqslant \ts 0 \\ x^{}_{1} \ts = \ts  s+t}} \nu^{}_{a^{}_{1}\, a \, a^{}_{2}\, a^{}_{3}} (s,\, t, \, x^{}_{2}) + \sum_{\substack{s,\, t \ts \geqslant \ts 0 \\ x^{}_{2} \ts = \ts  s+t}} \nu^{}_{a^{}_{1} \, a^{}_{2}\,a \, a^{}_{3}} (x^{}_{1},\,s,\, t) \Bigr), \]
and states the fact that the frequency of a given three-tile patch is determined by the sum of all frequencies of four-tile patches, three of which are the original ones. 
This can be seen as an analogy to the Kolmogorov consistency equation when extending a~measure, see \cite{Jan} for details. 

All the above-mentioned symmetries did \emph{not} involve the self-similar structure of $\vL$ (to be more precise, of $\underline{\vL}$), but self-similarity manifests itself on the level of patch frequencies significantly. Since the frequencies exist and are \emph{limiting} quantities (and not asymptotic ones), we can \emph{exactly} relate the frequency of a patch in the system and in its inflated version. This was first observed for pair-correlations (two-tile frequencies) in \cite{BGM19}. We extend this concept and obtain the following result, which is an extension of \cite[Thm.~3.19]{BGM19} 

\begin{theorem}
\label{thm:patch_renorm_subst}
Let $\underline{\vL}$ be as in Proposition \emph{\ref{prop:patch_freq_fun_properties}}, let $\lambda$ be the inflation factor, and let $(t^{}_{ij})^{}_{1\leqslant i,j \leqslant n}$ be the displacement matrix of the inflation rule defining $\underline{\vL}$. Then, the patch frequency functions satisfy the exact renormalisation relations\index{renormalisation equations!for patch frequencies} 
\begin{multline}
\label{eq:renorm}
    \nu^{}_{a^{}_{1}\, \cdots \, a^{}_{m}}(x^{}_{1}, \, \dots, \, x^{}_{m-1}) \\ \, = \, \myfrac{1}{\lambda} \sum_{k,\ell = 1}^m \sum_{\alpha^{}_{k} =1}^{n} \, \sum^{}_{r^{(k)} \in t^{}_{a^{}_{\ell}\alpha^{}_{k}}}  \nu^{}_{\alpha^{}_{1}\, \cdots \, \alpha^{}_{m}} \left(\myfrac{x^{}_{1}{+}r^{(1)}{-}r^{(2)}}{\lambda},\, \dots, \, \myfrac{x^{}_{m}{+}r^{(m-1)}{-}r^{(m)}}{\lambda} \right).
\end{multline}
\end{theorem}
\begin{proof}
Based on the definition of patch frequencies (Eq.~\eqref{eq:patch_freq_def}), we start with a radius $R>0$ and some $z\in \RR$ such that $B^{}_{R}(z)$ contains a patch of $m$ tiles $a^{}_{1}, \, \dots, \, a^{}_{m}$ within distances $x^{}_{1}, \, \dots, \, x^{}_{m-1}$. Since the hull $\XX(\vL)$ is aperiodic, we can profit from local recognisability\index{recognisability}, so every tile belongs to a unique level-1 supertile, say an inflated version of tile~$\alpha^{}_{k}$. At the level of control points, this means that there exists an element $\vL' \in \XX(\vL)$ and a ball of radius $\tfrac{R}{\lambda}$ such that there is a unique set of control points\index{set!of all control points} $\alpha^{}_{1}, \, \dots, \, \alpha^{}_{m-1}$ with 
\[| \alpha^{}_{i+1} - \alpha^{}_{i}| \, = \, \myfrac{x^{}_{i}+r^{(i)}-r^{(i+1)}}{\lambda}, \]
where $r^{(i)}$ and $r^{(i+1)}$ are elements of the displacement matrix $T$ and encode the exact positions of points $a^{}_{i}$ and $a^{}_{i+1}$, respectively, in the corresponding level-1 supertiles\index{supertile}. We illustrate this decomposition in Figure \ref{fig:renormalisation}. 
\begin{figure}[h]
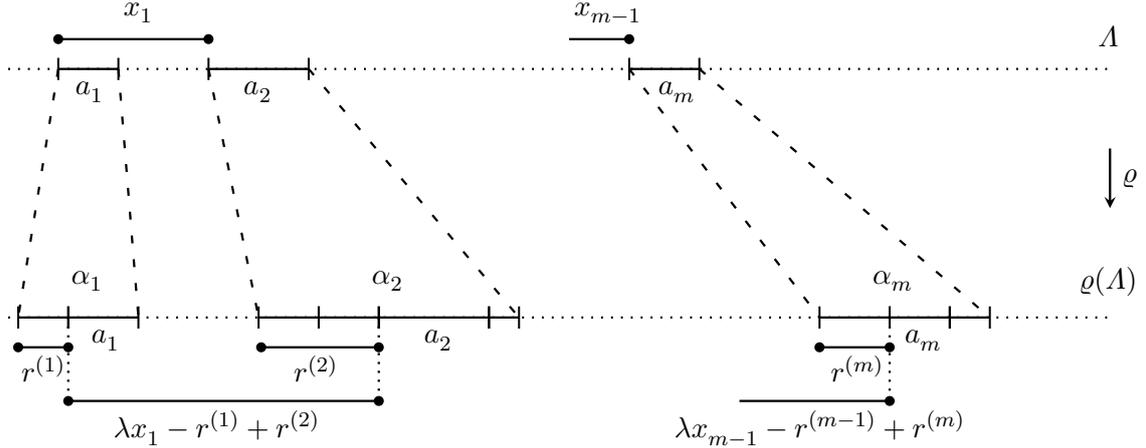

\centering
\include{renormalisation}
\vspace*{-14mm}
\caption[Renormalisation scheme for deriving the renormalisation equations]{Illustration of the renormalisation scheme. Every tile $a^{}_{i}$ in a patch in $\vL$ belongs to a unique level-1 supertile $\alpha^{}_{i}$, and within this supertile, it has a relative distance~$r^{(i)}$. This allows one to relate the patch frequencies at the level of supertiles, i.e., in $\varrho(\vL)$, with the patch frequencies in the original point set $\vL$.}
\label{fig:renormalisation}
\end{figure}
Now, summing over all possible supertiles that contain $a^{}_{1}, \, \dots, \, a^{}_{m}$ and all possible displacements within each supertile gives the claim together with the renormalising constant coming from changing the averaging radius in Eq.~\eqref{eq:patch_freq_def} from $R$ to $\tfrac{R}{\lambda}$. As the functions  $ \nu^{}_{a^{}_{1}\, \cdots \, a^{}_{m}}(x^{}_{1}, \, \dots, \, x^{}_{m-1})$ are independent of $\vL' \in \XX(\vL)$, we can relate the original and the rescaled version, though they may have come from different elements of the hull. 
\end{proof}

Observe that the arguments $\frac{x^{}_{i}+r^{(i)}-r^{(i+1)}}{\lambda}$ on the right hand side of Eq. \eqref{eq:renorm} are in general almost always smaller than $x^{}_{i}$ as $\lambda >1$. If this is not the case, it follows from the inequality $\frac{x^{}_{i}+r^{(i)}-r^{(i+1)}}{\lambda} \geqslant x^{}_{i}$ that the subset of all equations \eqref{eq:patch_freq_def} with all $x^{}_{i}$ satisfying 
\begin{equation}
    \label{eq:constraint}
    |x^{}_{i}| \, \leqslant \, \myfrac{1}{\lambda-1} \ \max^{}_{1\leqslant i \leqslant m-1} \max^{}_{1\leqslant j,\ell,p,s\leqslant n} \bigl\{ r^{(i)}-r^{(i+1)} \, \mid \, r^{(i)} \in t^{}_{s\ell}, \, r^{(i+1)} \in t^{}_{pj} \bigr\}
\end{equation}
forms a closed subsystem, which is finite. Indeed; we only consider $x^{}_{i} \in \vL^{}_{a^{}_{i+1}} - \vL^{}_{a^{}_{i}}$. 
Since the~$\vL^{}_{a}$ are all of finite local complexity, their difference is a locally finite set, so for all $i$, we have only finitely many $x^{}_{i}$ satisfying \eqref{eq:constraint}. Since $\cA$ is finite and $m<\infty$, we conclude the finiteness of the subsystem. We call it the \emph{self-consistent part}. The remaining equations form the \emph{recursive part}, the solution of which is uniquely determined by the self-consistent part. Indeed, if any coordinate exceeds the bound \eqref{eq:constraint}, applying \eqref{eq:renorm} repeatedly (together with the restriction $x^{}_{i} \in \vL^{}_{a^{}_{i+1}} \! \!- \vL^{}_{a^{}_{i}}$) reduces its modulus, until all coordinates in the expressions on the RHS of \eqref{eq:renorm} belong to the self-consistent part.

It remains to prove that the self-consistent part possesses a unique positive solution. We adapt  the approach from \cite{BGM19} using the normal form of a non-negative matrix (in the sense of Gantmacher \cite[Sec.~13.4]{Gant}), and we apply Lemma~3.18 from \cite{BGM19}, which we repeat here for convenience. 

\begin{lemma}[{\cite[Lemma~3.18]{BGM19}}]
\label{lem:normal-form}
Consider a non-negative matrix\/ $M$ in normal form\/
$M_{\mathsf{nf}}$, i.e. there exists a permutation similarity transformation bringing $M$ to the form 
\begin{equation}\label{eq:normal-form}
\renewcommand{\arraystretch}{1.2}
   M_{\mathsf{nf}} \, = \,
   \left(\begin{array}{ccc@{\;}c@{\;}c@{\;}c@{\;\,}c}
   M_{1} &   & \bs{0} &\vline&  && \\
     &  \ddots && \vline && \bs{0} & \\
   \bs{0} &  & M_{r} & \vline & & & \\
   \hline
   M_{r+1,1} & \cdots &  M_{r+1,r} && M_{r+1} && \bs{0} \\
    \vdots &  & && \!\!\! \ddots  & \ddots & \\
   M_{s,1} & &  \cdots && & M_{s,s-1} & M_{s}
   \end{array}\right)
\renewcommand{\arraystretch}{1}
\end{equation}
where $s\geqslant r \geqslant 1$ and all $M_{i}$ are
indecomposable, non-negative square matrices, and where, if $s>r$, each
sequence $M_{r+\ell,1}, \dots, M_{r+\ell,r+\ell-1}$ contains at least
one non-zero matrix.  
Further, let \/ $\lambda$ be the PF eigenvalue of\/ $M$. Then, $M$ has a corresponding
strictly positive eigenvector, meaning that all entries are positive, if and only if
\begin{enumerate}\itemsep2pt
\item $M_{i}$ has eigenvalue\/ $\lambda$ for every\/ $1\leqslant
  i\leqslant r$;
\item No\/ $M_{j}$ with\/ $r < j \leqslant s$ has\/ $\lambda$
  as an eigenvalue.
\end{enumerate}
In this situation, one has\/ $\lambda > 0$, and the eigenspace 
of\/ $\lambda$ is one-dimensional if and only if\/ $r=1$.  
\end{lemma}

\begin{theorem}\
\label{thm:uniqueness}
    The self-consistent part of Eq. \eqref{eq:renorm} possesses a unique positive solution. 
\end{theorem}
\begin{proof}
Consider now the vector $\bs{\nu}^{(m)}$ with entries $\nu^{}_{a^{}_{1}\, \cdots \, a^{}_{m}}(x^{}_{1}, \, \dots, \, x^{}_{m-1})$ with all $a^{}_{i}\in \cA$, $x^{}_{i} \in \vL^{}_{a^{}_{i+1}} \! \!- \vL^{}_{a^{}_{i}}$ and $x^{}_{i}$ satisfying \eqref{eq:constraint}. Then, \eqref{eq:renorm} can be seen as an eigenvalue problem for a~non-negative matrix $M^{(m)}$, in other words, the self-consistent part of  \eqref{eq:renorm} can be understood as 
\[\bs{\nu}^{(m)} \, = \, \myfrac{1}{\lambda} M^{(m)} \bs{\nu}^{(m)}. \]
The PF theorem, together with the strict ergodicity of the hull, implies that there exists a positive solution. To show uniqueness, we proceed as follows.

Observe first that $\nu^{}_{a\, \cdots \, a}(0, \, \dots, \,0) \, = \, \nu^{}_a$ is the letter frequency of $a\in\cA$. Part (iii) of Proposition \ref{prop:patch_freq_fun_properties} indicates that $M^{(m)}$ and $M^{(\ell)}$ with $\ell <m$ are related. In particular, one identifies the pair correlation functions $\nu^{}_{ij}(x)$ among the patch frequencies in $m-1$ different ways. 

By a suitable permutation of entries in $\bs{\nu}^{(m)}$, we can reach the following form of the matrix $M^{(m)}$, and we will show that this is its normal form.
\begin{align}
\label{eq:M}
M^{(m)} = \left(
        \begin{array}{ccccc}
        M^{(1)}\tmark{1} &  &  &  &\\[3pt]
        & \widetilde{M}^{(2)}\tmark{b} &  & &\\[3pt]    
        & & \ddots &&\\
        & & & \widetilde{M}^{(m-1)}\tmark{c} & \\[3pt]
        & & & & \ddots 
        \end{array}
    \right),
        \begin{tikzpicture}[remember picture,overlay]
            \draw ($(1)+(0.05,0.35)$) to ($(1)+(0.05,-0.2)$) to ($(1)-(0.97,0.2)$);
            \draw ($(b)+(0.05,1.01)$) to ($(b)+(0.05,-0.2)$) to ($(b)-(2.14,0.2)$);
            \draw ($(c)+(0.05,2.3)$) to ($(c)+(0.05,-0.2)$) to ($(c)-(4.59,0.2)$);
        \end{tikzpicture}
\end{align}
where $M^{}_{(1)} = M^{}_{\vrho}$ is the substitution matrix. Matrices $\widetilde{M}^{(\ell)}$ can be derived from the normal form of $M^{(\ell)}$ as follows. If $M^{(\ell)}$ has normal form $\left(\begin{smallmatrix}
    M^{\ell}_u & \bs{0} \\A^{(\ell)} & B^{(\ell)}
\end{smallmatrix} \right)$, then 
\begin{equation}
\label{eq:Mtilde}
    \widetilde{M}^{(\ell)} \, = \, \begin{pmatrix}
        M^{(\ell)}_u &&&&\\
        A^{(\ell)} & B^{(\ell)} && \bs{0}&\\ 
        A^{(\ell)} & \bs{0} & B^{(\ell)} &&\\
        \vdots &\vdots &\ddots& \ddots & \\
        A^{(\ell)} & \bs{0}&\cdots& \bs{0}& B^{(\ell)}
    \end{pmatrix}\begin{array}{c}
    \\
    \left.\hspace*{-20pt}\begin{tabular}{c}
    \, \\
    \, \\
    \, \\
    \,
    \end{tabular}\right\} \binom{m-1}{\ell-1}\, \mbox{times}
    \end{array}.
\end{equation}

Now, $M^{(1)} = M^{}_{\vrho}$ and $\bs{\nu}^{(1)} = \bs{\nu}^{}_{\vrho}$ is the right PF eigenvector with letter frequencies. Baake, Gähler and Ma\~{n}ibo showed that the assertion holds for $m=2$ \cite[Thm.~3.19]{BGM19}. In particular, they derived that $M^{(2)} = \left(\begin{smallmatrix}
    M^{}_{\vrho} & \bs{0} \\A^{(2)} & B^{(2)}
\end{smallmatrix} \right)$ is the normal form, and, by applying Lemma \ref{lem:normal-form}, they concluded that $\bs{\nu}^{(2)}$ is the only eigenvector corresponding to $\lambda$. 

We will proceed inductively. Suppose, for all $\ell < m$, $M^{(\ell)}$ has normal form $ \left(\begin{smallmatrix}
    M^{}_{\vrho} & \bs{0} \\A^{(\ell)} & B^{(\ell)}
\end{smallmatrix} \right)$, so the eigenspace corresponding to $\lambda$ is one-dimensional, determined by $\bs{\nu}^{(\ell)}$. 
Now, it is clear that $\widetilde{M}^{(\ell)}$ as stated in \eqref{eq:Mtilde} is already in its normal form. If $\left(
\bs{\nu}^{}_{\vrho},\ts  \widetilde{\bs{\nu}}^{(\ell)}
\right)^{T}$ is the PF eigenvector of $M^{(\ell)}$ corresponding to $\lambda$, then $\widetilde{\bs{\nu}} \, = \, \left(
\bs{\nu}^{}_{\vrho},\ts  \widetilde{\bs{\nu}}^{(\ell)}, \, \cdots, \ts  \widetilde{\bs{\nu}}^{(\ell)} \right)^{T}$ is a PF eigenvector of $\widetilde{M}^{(\ell)}$. 
If there were another linearly independent eigenvector, say $\bs{\mu}$, then the block structure of  $\widetilde{M}^{(\ell)}$ would imply $\bs{\mu} = \left(\bs{\nu}^{}_{\vrho}, \,  \ast \,  \right)^T$. If $\bs{\mu}$ belonged to a different eigenspace, $\widetilde{\bs{\nu}} - \bs{\mu}$ would again be an eigenvector of $\widetilde{M}^{(\ell)}$ to eigenvalue $\lambda$, this time of the form $\bigl(\bs{0}, \, \bs{x}^{}_{1},\, \cdots, \, \bs{x}^{}_{\binom{m-1}{\ell-1}} \bigr)$. 
From the block structure of $\widetilde{M}^{(\ell)}$, it follows that the components $\bs{x}^{}_{i}$ satisfy $B^{(\ell)}\bs{x}^{}_{i} \, =\, \lambda \bs{x}^{}_{i}$. But unless $\bs{x}^{}_i = \bs{0}$, this contradicts the fact that $M^{(\ell)}$ is in normal form, in particular, there are no eigenvectors with eigenvalue $\lambda$ to $B^{(\ell)}$. 

To finish the proof, we have to show that the last rows of $\eqref{eq:M}$ are all non-zero, which holds, and that the corresponding eigenspace for the PF eigenvalue $\lambda$ is one-dimensional. 
We show that $\nu^{}_{a^{}_{1}\, \cdots \, a^{}_{m}}(x^{}_{1}, \, \dots, \, x^{}_{m-1})$ with all $a^{}_{i}\in \cA$, $x^{}_{i} \in \vL^{}_{a^{}_{i+1}} \!\! - \vL^{}_{a^{}_{i}}$ and $x^{}_{i}\neq 0$ is coupled to $\nu^{}_{a\, \cdots \, a}(0, \, \dots, \,0)$ for some $a\in \cA$. This will connect the patch frequency with a letter frequency, effectively showing that the eigenvector $\bs{\nu}^{(\ell)}$ does not decouple and, thus, generates the entire eigenspace. 

The frequency $\nu^{}_{a^{}_{1}\, \cdots \, a^{}_{m}}(x^{}_{1}, \, \dots, \, x^{}_{m-1})$ corresponds to a patch of the form 
\[ T^{}_{a^{}_{1}} T^{}_{w^{}_{1}}T^{}_{a^{}_{2}} T^{}_{w^{}_{2}} \cdots T^{}_{a^{}_{m-1}} T^{}_{w^{}_{m-1}}T^{}_{a^{}_{m}}, \]
where $T^{}_{w^{}_{i}}$ denotes an arbitrary patch such that the length of $T^{}_{a^{}_{i}} T^{}_{w^{}_{i}}$ equals $x^{}_{i}$. In the symbolic setting, this corresponds to a word 
$u \, = \, a^{}_{1}\,w^{}_{1}\,a^{}_{2}\, w^{}_{2} \cdots a^{}_{m-1}\, w^{}_{m-1}\, a^{}_{m}$. 
Primitivity of $\vrho$ implies the existence of a minimal power $j\in \NN$ of $\vrho$ and a letter $a\in \cA$ such that $u$ is a subword of $\vrho^{j}(a)$, showing the coupling, as $\bigl(M^{(m)}\bigr)^j$ contains a non-zero element at the desired position.
\end{proof}

Suppose further that the control points $\underline{\vL}$ possess a model set description (for example, if $\vrho$ is a binary Pisot substitution \cite{HolSol}), meaning that there exists a cut-and-project scheme $(\RR, \, H, \, \cL)$ with a locally compact abelian group $H$ with Haar measure $\mu^{}_{H}$, natural projections $\pi, \, \pi^{}_{\mathrm{int}}$ and a star map $\star:\pi(\cL)\to H$, and relatively compact, topologically regular windows $\varOmega^{}_{i} \subset H$ such that $\vL^{}_{i}$ and $\oplam(\varOmega^{}_{i}) \defeq \left\{ \pi(x) \ \mid \ x\in \cL, \ \pi^{}_{\mathrm{int}}(x)\in \varOmega^{}_{i} \right\}$ differ at most by a set of zero density (for more details on cut-and-project sets and model sets and the standard notation, see \cite[Sec.~7]{BGr} or \cite{Sing}). Then, we have an alternative description of the patch frequency functions using the equidistribution results for model sets.

\begin{prop}
\label{prop:freq_CPS_wind}
For any primitive inflation substitution $\vrho$ with a model set description as above, the patch frequency function $\nu^{}_{a^{}_{1}\, \cdots \, a^{}_{m}}(x^{}_{1}, \, \dots, \, x^{}_{m-1})$ can be expressed, for all  $a\in\cA$, $x^{}_i \in \vL^{}_{a^{}_{i+1}} - \vL^{}_{a^{}_{i}}$ and all \/$m\geqslant 2$, as 
\begin{align*}
    \nu^{}_{a^{}_{1}\, \cdots \, a^{}_{m}}(x^{}_{1}, \, \dots, \, x^{}_{m-1}) & \, = \, 
    \myfrac{1}{\vol(\varOmega)} \int_H \prod_{j=1}^m \boldsymbol{1}^{}_{\varOmega^{}_{a^{}_j}} \Bigl(y + \sum_{i=0}^{j-1} x^{\star}_i \Bigr) \dd\mu^{}_{H}(y)\\[3pt]
    & \, = \, \myfrac{1}{\vol(\varOmega)}\,\vol\Bigl({\displaystyle \bigcap_{j=1}^m} \bigl(\varOmega^{}_{a^{}_{j}} -\sum_{i=0}^{j-1} x^{\star}_i \bigr) \Bigr) \, = \, g^{}_{a^{}_{1}\, \cdots \, a^{}_{m}}(x^{\star}_{1}, \, \dots, \, x^{\star}_{m-1}),
    \end{align*}
where we set $x^{}_{0} =0$, $\bs{1}^{}_{A}$ denotes the characteristic function of a set $A$ and $\varOmega = \bigcup^{}_{a\in\cA} \varOmega^{}_{a}$. The functions 
\[g^{}_{a^{}_{1}\, \cdots \, a^{}_{m}}\! : H^{m-1}  \, \longrightarrow \, \RR^{}_{\geqslant 0}\]
are compactly supported continuous functions. 
\end{prop}
\begin{proof}
The representation as the integral is a standard result using the equidistribution results, see \cite{Moody02}. Since $\boldsymbol{1}^{}_{\varOmega^{}_{j}+y}$ are $L^1$ and $L^{\infty}$ functions with a compact support (as $\varOmega^{}_{i}$ are relatively compact) for all $j$ and all $y\in H$, the same holds for their finite product, whose integral is a continuous function w.r.t. all $x^{\star}_i$ by the standard analysis argument. Since the windows are compact, the set of translations for which two windows have a non-empty intersection is also compact. Now, as the intersection of two compact sets is also compact, we can proceed inductively and conclude that, after finitely many steps, the set of all mutual translations for which $m$ windows do have a non-empty intersection is a compact set in $H^{m-1}$, which finishes the proof. 
\end{proof}

\begin{remark}
For single-component Euclidean model sets (i.e. $H = \RR^d$, for some $d\geqslant1$), the functions $g$ were considered in \cite{DM09} (with a slightly different choice of arguments) under the name $\mathcal{I}^{(m-1)}$. The two- and three-tile patch frequencies (of two and three point correlations) are, in this setting, sufficient to reconstruct the set $\vL$ up to a set of zero density \cite[Thm.~3]{DM09}. We note that this is no longer true once we step beyond the Euclidean setting \cite{DM12}.  

We also mention that the functions $g$ are Fourier transformable. If $H = \RR^d$, a direct computation yields 
\[
    \widehat{\, g^{}_{a^{}_{1}\, \cdots \, a^{}_{m}}\,} (k^{}_{1}, \, \dots, \, k^{}_{m-1}) \, = \,  \myfrac{1}{\vol(\varOmega)} \, \widehat{\bs{1}^{}_{\varOmega^{}_{a^{}_{1}}}}(k^{}_{1}) \, \widehat{\,\bs{1}^{}_{\varOmega^{}_{a^{}_{m}}}\,} (k^{}_{m-1}) \prod_{j=2}^{m-2} \widehat{\,\bs{1}^{}_{\varOmega^{}_{a^{}_{j}}}\,}(k^{}_{j-1}-k^{}_{j}).
\]
The question of Fourier transformability of these functions is natural and could be the first step towards a finer spectral decomposition of the spectral measure of the underlying dynamical system via refinement of the pair-correlation measures (introduced using the pair-correlation functions) in \cite{BGM19, NeilDiss}. These pair-correlation measures are Fourier transformable, and the proof of this statement relies on their representation via Eberlein convolution. A first attempt to define a patch frequency measure and provide a generalisation of the Eberlein convolution, which would fit into this setting and framework, can be found in \cite{Jan}. \exend
\end{remark}

\section{From self-similar to symbolic settings}
This last section explains how to pass from the geometric, self-similar description of one-dimensional tilings to other geometric descriptions and, in particular, back to the symbolic setting. We demonstrate this approach first for the two-tile frequencies in the Fibonacci case, and then state the general results. 

\subsection{Fibonacci example}
Consider the usual self-similar Fibonacci tiling\index{substitution!Fibonacci} with prototiles of length $\tau$ and $1$ arising from the substitution rule $\vrho = (a\mapsto ab, b\mapsto a)$. In order to derive the relative frequencies (pair correlations) for the symbolic case (to be understood as a tiling with prototiles of the same length $1$ and different labels $a$, $b$), we consider the pair-correlation functions in the geometric setting (two-tile frequencies) $\nu^{}_{\alpha \, \beta}$ satisfying the following set of renormalisation equations, which first appeared in \cite{BG-renorm}.
\begin{equation}
\label{eq:Fibo_corr_renorm}
\begin{split}
   \nu^{}_{aa} (x) \, & = \, \myfrac{1}{\tau} \Big(
     \nu^{}_{aa} \bigl(\tfrac{x}{\tau}\bigr) + \nu^{}_{ab} \bigl(\tfrac{x}{\tau}\bigr)
    +\nu^{}_{ba} \bigl(\tfrac{x}{\tau}\bigr) + \nu^{}_{bb} \bigl(\tfrac{x}{\tau}\bigr)
       \Big) , \\[2pt]
   \nu^{}_{ab} (x) \, & = \, \myfrac{1}{\tau} \Big(
      \nu^{}_{aa} \bigl(\tfrac{x-\tau}{\tau} \bigr) + 
      \nu^{}_{ba} \bigl(\tfrac{x-\tau}{\tau} \bigr)  \Big) , \\[2pt]
   \nu^{}_{ba} (x) \, & = \, \myfrac{1}{\tau} \Big(
      \nu^{}_{aa} \bigl(\tfrac{x+\tau}{\tau} \bigr) + 
      \nu^{}_{ab} \bigl(\tfrac{x+\tau}{\tau} \bigr)  \Big), \\[2pt]
   \nu^{}_{bb} (x) \, & = \, \myfrac{1}{\tau} \ts 
      \nu^{}_{aa} \bigl(\tfrac{x}{\tau}\bigr),
\end{split}
\end{equation}

We denote the pair correlations in the symbolic case by $\nu^{\mathrm{s}}_{\alpha \beta}$. By choosing the tile lengths to be $1$, we immediately get that $\nu^{\mathrm{s}}_{\alpha \beta}(m)$ can only be non-zero for $m\in \ZZ$. 

Let $m\in \NN$. Then, $\nu^{\mathrm{s}}_{\alpha \beta}(m)$ is the relative frequency of a cylinder set consisting of $m+1$ tiles with the first tile being  $\alpha$ and the last $\beta$. 
To obtain $\nu^{\mathrm{s}}_{\alpha \beta}(m)$, add all relative frequencies of all patches (in the self-similar case) consisting of tiles $\alpha$ on the left and~$\beta$ on the right, such that there are exactly $m-1$ tiles between them. 
The number of different tiles of type $\alpha$ and~$\beta$ then gives the distances, for which one has to take the $\nu^{}_{\alpha\beta}$ into account. 

We demonstrate this for $\nu^{\mathrm{s}}_{\alpha \beta}(4)$. Since the Fibonacci chain is Sturmian \cite{Loth}, there are exactly 6 different legal subwords of length 5, namely $aabaa, \ aabab,\ baaba, \ ababa,\ abaab, \ babaa $. 
No legal subword starting and ending with $b$ implies $\nu^{\mathrm{s}}_{bb}(4)=0$.  Now, for $\nu^{\mathrm{s}}_{aa}(4)$, observe that the two subwords starting and ending with $a$ can be identified with patches (using the natural tile lengths $\tau$ and $1$), where the distance between the first and last left-end control point is $3\tau +1$ ($aabaa$), and $2\tau+2$ ($ababa$). This yields 
\[\nu^{\mathrm{s}}_{aa}(4) \, = \, \nu^{}_{aa}(3\tau+1) + \nu^{}_{aa}(2\tau+2) = \tau^{-3}. \]

We summarise these observations as follows. 

\begin{prop}
\label{prop:fibo_symb_pair}
Let $\nu^{}_{\alpha\beta}$ denote the pair correlations of the Fibonacci tiling and let $\nu^{\mathrm{s}}_{\alpha \beta}$ denote those of the symbolic one. Then, for any $m\in \NN$ and any $\beta \in \{a,b\}$, one has 
\begin{equation}
\label{eq:fibo_symb_vs_geom}
\nu^{\mathrm{s}}_{a\beta}(m) = \sum_{n=1}^m \nu^{}_{a\beta}(n\tau + m-n) \qquad \mbox{and} \qquad \nu^{\mathrm{s}}_{b\beta}(m) = \sum_{n=0}^{m-1} \nu^{}_{b\beta}(n\tau + m-n). 
\end{equation}
Moreover, we have $\nu^{\mathrm{s}}_{ab}(m) = \nu^{\mathrm{s}}_{ba}(m)$ and $\sum_{\alpha,\, \beta} \nu^{\mathrm{s}}_{\alpha\beta}(m) =1$ for any $m\in \ZZ$. 
\end{prop}
\begin{proof}
The frequencies $\nu^{\mathrm{s}}_{\alpha \beta}$ exist uniformly due to unique ergodicity, and the interplay between the geometric and symbolic description has been explained above. The different summations reflect that, in the case of $\nu^{\mathrm{s}}_{a\beta}$, one has to include at least one $a$-tile, whereas for the $\nu^{\mathrm{s}}_{b\beta}$, one has to count at least one tile $b$. The symmetry relation follows from palindromicity of the \emph{symbolic} Fibonacci hull; consult \cite[Sec.~4.3]{BGr}. The last claim says that the relative frequency of two arbitrary points of the integer lattice $\ZZ$ within distance $m\in \ZZ$ is always 1, which is clearly true. 
\end{proof}

Since the set of control points of the geometric Fibonacci tiling is a model set arsing from a cut-and-project scheme $(\RR,\, \RR,\, \cL)$, $\cL = \langle (1,1)^{T}, \, (\tau,\, 1-\tau)^{T} \rangle^{}_{\ZZ}$ via a window $\varOmega=[-1, \, \tau)$, with $\tau = \myfrac{1+\sqrt{5}}{2}$ (for details see \cite[Sec.~7]{BGr}), we obtain a stronger result. 

\begin{lemma}
\label{lem:fibo_corr_restriction_symb}
For the relative frequencies of the symbolic Fibonacci chain $\nu^{\mathrm{s}}_{\alpha \beta}$, we have 
\[ \nu^{\mathrm{s}}_{\alpha \beta}(m) \, = \, \nu^{}_{\alpha \beta} \bigl( \bigl\lfloor \tfrac{m}{\tau} \bigr \rfloor + m - \bigl\lfloor \tfrac{m}{\tau} \bigr \rfloor \bigr) + \nu^{}_{\alpha \beta}\bigl( \bigl\lceil \tfrac{m}{\tau} \bigr \rceil + m - \bigl\lceil \tfrac{m}{\tau} \bigr \rceil \bigr). \]
\end{lemma}
\begin{proof}
Since $\nu_{\alpha \beta}(x) \neq 0$ only if $z\in \vL - \vL = \oplam([-\tau^2, \, \tau^2])$, we can restrict the possible $n$ for which $\nu^{}_{\alpha\beta}(n\tau + m-n)$ can be non-zero, and reduce the number of summands in Eq.~\eqref{eq:fibo_symb_vs_geom}. For that, consider $(n\tau + m-n)^{\star} = n(1-\tau) + m -n \in [-\tau^2, \, \tau^2]$, which implies 
\[ \tfrac{m}{\tau} -1 \, \leqslant \,  n \, \leqslant \, \tfrac{m}{\tau} +1. \]
As both $m,n \in \ZZ$ and since $\tfrac{m}{\tau} \notin \ZZ$ for non-zero $m$, we have $n \in \bigl\{ \lfloor \tfrac{m}{\tau} \bigr \rfloor, \, \lceil \tfrac{m}{\tau} \bigr \rceil \bigr\}$.
\end{proof}

\subsection{General results}
Note first that we used the fact that the natural tile lengths of the geometric realisation of the Fibonacci tiling are rationally independent\index{rational independence}, which allowed us to identify the total number of tiles within a distance given by a number $n\tau + m$. To ensure that the same holds for a primitive substitution $\varrho$, we need to assume that the characteristic polynomial\index{characteristic polynomial} of $M^{}_{\varrho}$ is irreducible (which is a condition appearing quite often; compare, for example, \cite[Prop.~2.17]{Yaari})

\begin{lemma}
\label{lem:rat_indep_lengths}
Let $M \in \ZZ^{n\times n}$ be a non-negative irreducible matrix with irreducible characteristic polynomial. Let further $\bs{v}$ be its PF eigenvector corresponding to the leading eigenvalue $\lambda$. Then, the entries of $v$ are rationally independent. 
\end{lemma}
\begin{proof}
Since $M\bs{v} = \lambda \bs{v}$, we can assume that all $v^{}_{i} \in \ZZ[\lambda]$. Since $\lambda$ is a root of monic irreducible polynomial of degree $n$, $\ZZ[\lambda]$ is a $\ZZ$-module of rank $n$. We show that $\langle v^{}_{1}, \, \dots, \, v^{}_{n} \rangle^{}_{\ZZ}$ is also a $\ZZ$-module of rank $n$, which is equivalent to saying that the entries $v^{}_{i}$ are rationally independent. 

As $\bs{v}$ is the PF eigenvector, $v^{}_{1} \neq 0$. Since $M\bs{v} = \lambda \bs{v}$, we have $\lambda v^{}_{i} \in \langle v^{}_{1}, \, \dots, \, v^{}_{n} \rangle^{}_{\ZZ}$ for all $i$, and by linearity, we obtain that $ x \in \langle v^{}_{1}, \, \dots, \, v^{}_{n} \rangle^{}_{\ZZ}$ implies $\lambda x \in \langle v^{}_{1}, \, \dots, \, v^{}_{n} \rangle^{}_{\ZZ}$. In particular, 
\[ \langle v^{}_{1}, \, \lambda v^{}_{1}, \, \dots, \, \lambda^{n-1}v^{}_{1} \rangle^{}_{\ZZ} \, \subseteq \, \langle v^{}_{1}, \, \dots, \, v^{}_{n} \rangle^{}_{\ZZ} \, \subseteq \, \ZZ[\lambda].   \]
Since $1,\, \lambda, \, \dots, \, \lambda^{n-1}$ are linearly independent over $\QQ$ as the minimal polynomial of $\lambda$ coincides with the characteristic polynomial of $M$ (which is of order $n$), $v^{}_{1},\, \lambda v^{}_{1}, \, \dots, \, \lambda^{n-1}v^{}_{1}$ are also rationally independent. Consequently, $\langle v^{}_{1},\, \lambda v^{}_{1}, \, \dots, \, \lambda^{n-1}v^{}_{1} \rangle^{}_{\ZZ}$ is a $\ZZ$-module of rank $n$ and so is $\langle v^{}_{1}, \, \dots, \, v^{}_{n} \rangle^{}_{\ZZ}$. 
\end{proof}

The same property also holds for the left eigenvector. In our setting, with a primitive substitution~$\varrho$ over $\cA^{}$ with $\lambda$ being an algebraic integer of order $n$ (equivalent to the irreducibility of the characteristic polynomial of $M^{}_{\varrho}$), the lemma implies that the natural tile lengths  $\ell^{}_{i}$ are rationally independent.  

\begin{theorem}
\label{thm:geom_to_symb}
Let $\varrho$ be a primitive substitution over $\cA^{}$ such that the characteristic polynomial of $M^{}_{\vrho}$ is irreducible. Let $\ell^{}_{j}$ denote the natural tile lengths of the geometric inflation rule induced by $\varrho$. Further, let $\nu^{}_{\alpha \, \beta}$ denote the pair correlation of the self-similar tiling. Then, for any element of the symbolic hull $\XX(\vrho)$, the pair correlations $\nu^{\mathrm{s}}_{\alpha \beta}$ exists uniformly and, for all $a^{}_{i}, \, \beta \in \cA^{}_n$ and for all $m\in \NN$, we have 
\begin{equation}
\label{eq:pair_corr_symb_from_geom}
    \nu^{\mathrm{s}}_{a^{}_{i}  \beta} \, (m) \, = \, \sum_{\substack{m^{}_{j} \ts \geqslant \ts  0, \, m^{}_i \ts \neq \ts 0 \\ \sum m^{}_j \, =\,  m}} \nu^{}_{a^{}_{i} \beta} \Bigl( \sum_{j=1}^n m^{}_{j}\ell^{}_{j} \Bigr).
\end{equation}
\end{theorem}
\begin{proof}
The relative frequencies $\nu^{\mathrm{s}}_{a^{}_{i}\beta}(m)$ of a two-letter patch consisting of letters $a^{}_{i}$ and $\beta$ within distance $m$ are given as the sum of relative frequencies of all subwords of length $m+1$ starting with $a^{}_{i}$ and ending with $\beta$. All these frequencies exist uniformly due to unique ergodicity. 
Now, the induced geometric inflation gives a tiling with prototiles\index{prototile} of length $\ell^{}_{j}$, all rationally independent by Lemma \ref{lem:rat_indep_lengths}. Since $\nu^{}_{a^{}_i\beta}(x)$ can only be non-zero if $x$ is a distance between two tiles, we can assume that $x>0$ and that $x$ can be uniquely decomposed as $x = \sum_{j=1}^n m^{}_{j} \ell^{}_{j}$, where $m^{}_{j} \in \NN^{}_{0}$ and $m^{}_{i} >0$. This follows from the rational independence of~$\ell^{}_{j}$'s. The quantity $\nu^{}_{a^{}_i\beta}(x)$ then represents the sum of all patches consisting of $m+1$ tiles starting with $a^{}_{i}$, ending with tile $\beta$. This gives the desired correspondence.    
By summing over all subwords of length $m+1$ starting with $a^{}_{i}$ and ending with $\beta$, the correspondence yields~\eqref{eq:pair_corr_symb_from_geom}. 
\end{proof}

\begin{remark}
If the geometric inflation rule $\varrho$ permits a model set\index{model set} description with total window $\varOmega$, we can restrict the summation on the RHS of Eq.~\eqref{eq:pair_corr_symb_from_geom} (analogously to Lemma~\ref{lem:fibo_corr_restriction_symb}) by taking into account only those summands that can be non-zero and obtain
\[\nu^{\mathrm{s}}_{a^{}_{i}  \beta} \, (m) \, = \, \sum_{\substack{m^{}_{j} \ts \geqslant \ts  0, \, m^{}_i \ts \neq \ts 0 \\ \sum m^{}_j \, =\,  m \\  \sum_{j=1}^n m^{}_{j}\ell^{\star}_{j}\ts \in \ts \varOmega - \varOmega}} \nu^{}_{a^{}_{i} \beta} \Bigl( \sum_{j=1}^n m^{}_{j}\ell^{}_{j} \Bigr). \]
\end{remark}

In the proof, we used the fact that $\XX(\varrho)$ can also be understood as a suspension flow of $(\XX(\varrho),\, \ZZ)$ with the constant roof (height) function. Using a different roof function (choosing tile lengths) $\boldsymbol{\ell}' = (\ell^{\prime}_1,\, \ell^{\prime}_2,\, \dots,\, \ell^{\prime}_n )$ and considering the suspension flow $(\XX^{(\boldsymbol{\ell^{\prime}})}(\vrho),\, \RR)$, we obtain for the patch frequency measure associated with any tiling from $\XX^{(\boldsymbol{\ell^{\prime}})}(\vrho)$ an analogous claim. Its proof mimics the previous one and relies on the fact that we can interpret the value $\nu^{}_{\alpha \beta}(z)$ as the sum of relative frequencies of all cylinders of \emph{one} length. 

\begin{theorem}
\label{thm:geom_to_arb}
Let $\varrho$ and $\nu^{}_{\alpha \beta}$ be as in Theorem \emph{\ref{thm:geom_to_symb}}. Consider a hull of a  tiling arising from $\XX(\varrho)$ using tile lengths $\boldsymbol{\ell}' = (\ell^{\prime}_1,\, \ell^{\prime}_2,\, \dots,\, \ell^{\prime}_n )$, i.e., consider a suspension flow $(\XX^{(\boldsymbol{\ell^{\prime}})}(\vrho),\,\RR)$ with the height function $\boldsymbol{\ell}^{\prime}$. Then, for the pair correlation $\nu^{\prime}_{\alpha \beta}$ on $\XX^{(\boldsymbol{\ell^{\prime}})}(\vrho)$, the following holds.
\begin{enumerate}
    \item[(i)] If the $\ell'_{j}$ are rationally independent, we have 
    \[ \nu^{\prime}_{\alpha \beta} \Bigl(\sum_{j=1}^n m^{}_{j}\ell^{\prime}_{j}\Bigr) \, = \, \nu^{}_{\alpha \beta} \Bigl(\sum_{j=1}^n m^{}_{j}\ell^{}_{j}\Bigr). \]
    \item[(ii)] If the $\ell'_{j}$ are rationally dependent, we have 
    \[\nu^{\prime}_{a^{}_{i}  \beta} \, \Bigl( \sum_{j=1}^n m^{\prime}_{j}\ell^{\prime}_{j} \Bigr) \, = \, \sum_{\substack{m^{}_{j} \ts \geqslant \ts  0, \, m^{}_i \ts \neq \ts 0 \\ \sum m^{}_j \, =\,  \sum m^{\prime}_j }} \nu^{}_{a^{}_{i} \beta} \Bigl( \sum_{j=1}^n m^{}_{j}\ell^{}_{j} \Bigr) \]
    for $1\leqslant i \leqslant n $ and for all $m^{\prime}_j \in \NN^{}_0$ with $m^{\prime}_i \neq 0$. \qed
\end{enumerate}
\end{theorem}

The same approach applies to any patch frequency functions, and we end up with a method that can decide whether or not a given patch is legal for a given tiling (bi-infinite sequence) and give the exact relative frequency in the first case. For example, under the assumptions of Theorem \ref{thm:geom_to_symb}, for the relative frequencies on the symbolic hull\index{hull!symbolic} $\XX(\varrho)$, one has
\[\nu^{\mathrm{s}}_{a^{}_{1}\, \cdots \, a^{}_{k}}(m^{}_{1}, \, \dots, \, m^{}_{k-1}) \, = \, \sum_{\substack{m^{(i)}_{j} \ts \geqslant \ts  0, \, m^{(i)}_i \ts \neq \ts 0 \\ \sum m^{(i)}_j \, =\,  m^{}_{i}}} \nu^{}_{a^{}_{1}\, \cdots \, a^{}_{k}}\Bigl( \sum_{j=1}^n m^{(1)}_{j}\ell^{}_{j}, \, \dots , \,  \sum_{j=1}^n m^{(k-1)}_{j}\ell^{}_{j}\Bigr).  \]
It is obvious how to extend Theorem \ref{thm:geom_to_arb}, but we will not write down the statement as the notation would be even harder to read, and the underlying principle should already be clear.

\section*{Acknowledgments}
It is a pleasure to thank Michael Baake and Neil Ma\~nibo for all helpful discussions and for their comments to the first draft of this paper. This work was supported by the German Research Council (Deutsche
Forschungsgemeinschaft, DFG) under CRC 1283/2 (2021 - 317210226).

\end{document}